\documentclass[a4paper,12pt,oneside,reqno ]{amsart}

\usepackage[T1]{fontenc}
\usepackage[utf8]{inputenc}
\usepackage{lmodern}
\usepackage[english]{babel}

\usepackage[%
left=2.5cm,       
right=2.5cm,      
top=3.5cm,        
bottom=3.5cm,     
heightrounded,    
bindingoffset=0mm 
]{geometry}
\usepackage{amsrefs}
\usepackage{amsmath}
\usepackage{amssymb}
\usepackage{mathtools}
\usepackage{mathrsfs}
\usepackage[colorlinks=true,linkcolor=blue,urlcolor=blue,citecolor=blue]{hyperref}

\numberwithin{equation}{section}

\usepackage{hyperref} 

\usepackage{tikz}
\usepackage{graphicx}
\usepackage{xcolor}
\usepackage[font=small]{caption}

\usepackage{bookmark}

\theoremstyle{plain}
\newtheorem{theorem}{Theorem}[section]
\newtheorem{lemma}[theorem]{Lemma}
\newtheorem{proposition}[theorem]{Proposition}

\newtheorem*{mtheorem}{Main Theorem}
\theoremstyle{definition}

\newtheorem{remark}[theorem]{Remark}
\newtheorem{open.problem}[theorem]{Open Problem}

\newcommand{\Leb}[1]{\mathscr{L}^{#1}} 
\newcommand{\Div}{\operatorname{div}}
\newcommand{\N}{\mathbb{N}}
\newcommand{\R}{\mathbb{R}}

\title[Quantitative Fractional Gaussian Isoperimetric Inequality]{A quantitative dimension free 
isoperimetric inequality for the fractional Gaussian perimeter}
\author[A.\ Carbotti]{Alessandro Carbotti}
\address{Dipartimento di Matematica
	e Fisica ``E. De Giorgi'', Universit\`a del Salento,
	Via Per Arnesano, 73100 Lecce, Italy.}
\email{alessandro.carbotti@unisalento.it}
\author[S.\ Cito]{Simone Cito}
\address{Dipartimento di Matematica
	e Fisica ``E. De Giorgi'', Universit\`a del Salento,
	Via Per Arnesano, 73100 Lecce, Italy.}
\email{simone.cito@unisalento.it}
\author[D. A. \ La Manna]{Domenico Angelo La Manna}
\address{Dipartimento di Matematica e Applicazioni ``R. Caccioppoli'', Universit\`a degli studi di Napoli
``Federico II'', Complesso di Monte Sant'Angelo, Via Cintia, 80126 Naples, Italy.}
\email{domenicoangelo.lamanna@unina.it}
\author[D.\ Pallara]{Diego Pallara}
\address{Dipartimento di Matematica
	e Fisica ``E. De Giorgi'', Universit\`a del Salento, and INFN, Sezione di Lecce,
	Via Per Arnesano, 73100 Lecce, Italy.}
\email{diego.pallara@unisalento.it}

\date{\today}  

\keywords{Fractional Perimeters, Fractional Ornstein-Uhlenbeck, Extension Methods, Isoperimetric Inequalities, Stability Inequalities}

\subjclass[2010]{35R11, 49Q20}

\begin{document}
	\begin{abstract}
		We prove a quantitative isoperimetric inequality for the fractional Gaussian perimeter using extension techniques. Though the exponent of the Fraenkel asymmetry is not sharp, the constant appearing in the inequality does not depend on the dimension but only on the Gaussian volume of the set and on the fractional order.
	\end{abstract}
	
	\maketitle
	
	\tableofcontents
	
	\section{Introduction}
The Gaussian isoperimetric inequality states that among all sets with prescribed Gaussian measure, the halfspace is the one with least Gaussian perimeter. This result has been proved independently by Borell \cite{borell} and Sudakov-Tsirelson \cite{SudCir}. In \cite{CarKer} it has been proved that halfspaces are the only volume-constrained minimizers for the Gaussian perimeter, while in \cite{CiFuMaPr, BarBraJul, BarJul} inequalities of quantitative type, that allow to relate the deficit between a halfspace and a set with the same Gaussian volume with some function of the Gaussian measure of their symmetric difference, are proved. The results in \cite{CiFuMaPr} have been improved in \cite{MosNee1,MosNee2}.
On the other side, fractional perimeters and nonlocal perimeters depending on more general kernels have been object of great attention in the last years, since they are related to nonlocal minimal surfaces \cite{CafRoqSav, MazRosTol}, phase transitions \cite{valdinoci}, fractal sets \cite{lombardini} and many other problems. In the Euclidean setting, fractional isoperimetric inequalities of qualitative and quantitative type have been proved in \cite{CesNov, FraSei} and \cite{FuMiMo, FiFuMaMiMo}, respectively. See also \cite{ComSte} where the authors introduce a notion of fractional perimeter using a distributional approach and \cite{DiNoRuVa} where an isoperimetric problem with the competition of two fractional perimeters of different order is studied.
In \cite{NovPalSir} the authors introduce a notion of fractional Gaussian perimeter using the by now well known extension techniques introduced in \cite{CafSil, StiTor} and they prove a qualitative isoperimetric inequality in the more general setting of abstract Wiener spaces.
Inspired by the paper \cite{BraCinVit}, where the authors prove a stability estimate for the fractional Faber-Krahn inequality, and taking into account the extension technique of \cite{StiTor}, we prove a quantitative isoperimetric inequality for a fractional perimeter in the Gauss space. Although the technique is similar, we find a different exponent since the perimeter is given by the $H^{s/2}$ norm of the characteristic function, while the first eigenvalue depends on the $H^s$ norm. Moreover, similarly to the local case (see \cite{BarBraJul, eldan}), the constant appearing in the inequality does not depend on the dimension of the ambient space. This fact exploits Proposition \ref{prop:halfspace} where we prove that halfspaces have the same fractional Gaussian perimeter as halflines having the same one dimensional Gaussian measure. To conclude, we notice that the asymptotics as $s\to 0^+$ under the pointwise convergence and the asymptotics as $s\to 1^-$ under $\Gamma$-convergence have been studied in \cite{CCLP2} and in \cite{CCLP1} in the present setting. In \cite{DL} the authors give a different notion of Gaussian fractional perimeter of a measurable set $E$ in a bounded domain $\Omega\subset\R^N$ using a singular integral representation of the form
\begin{equation*}
\begin{split}
P^\gamma_s(E;\Omega):&=\int_{E\cap\Omega}e^{-\frac{|x|^2}{4}}dx\int_{E^c\cap\Omega}\frac{e^{-\frac{|y|^2}{4}}}{|x-y|^{N+s}}dy+\int_{E\cap\Omega}e^{-\frac{|x|^2}{4}}dx\int_{E^c\cap\Omega^c}\frac{e^{-\frac{|y|^2}{4}}}{|x-y|^{N+s}}dy \\
&+\int_{E\cap\Omega^c}e^{-\frac{|x|^2}{4}}dx\int_{E^c\cap\Omega}\frac{e^{-\frac{|y|^2}{4}}}{|x-y|^{N+s}}dy,
\end{split}
\end{equation*}
and they prove the Gamma convergence of $(1-s)P^\gamma_s(E;\Omega)$ to the Gaussian perimeter as $s\to 1^-$ exploiting techniques similar to the ones used in \cite{AmDeMa}. See also \cite{BerPal}, where kernels with faster than $L^1$ decay at infinity are taken into account.

The precise statement of our main result is the following.
\begin{mtheorem}
	\label{teo:maintheorem}
	Let $N\geq 1$, $s\in (0,1)$ and $m\in (0,1)$. For any set $E$ with finite fractional Gaussian perimeter
	of order $s$ and $\gamma(E)=m$ we have
	\begin{equation} \label{eq:main}
	D_s^\gamma (E):=P_s^\gamma (E)- P^\gamma_s(H) \geq  C_{s,m} \mathcal{A}_\gamma(E)^{\frac 2s},
	\end{equation}
	where $H$ is any halfspace with $\gamma(H)=\gamma (E)$ and $C_{s,m}$ is a positive constant which depends only on $s$ and $m$.
\end{mtheorem}
Here $\mathcal{A}_\gamma(E)$ denotes the {\em Gaussian Fraenkel asymmetry}: for the precise definition of the quantities involved in \eqref{eq:main} we invite the reader to check Section \ref{sec:intro}.
We notice that, as far as we know, the notion of perimeter used here is not a particular case of the one given in \cite{cabre, pagliari}, where the authors independently prove the local minimality of halfspaces for a broad class of nonlocal perimeters using some calibration methods. See also the recent \cite{CaDoPaPi} where the result is proved in the more general setting of Carnot Groups.

The paper is structured as follows. In Section \ref{sec:intro} we introduce the notation used throughout the paper and state some preliminary results. In Section \ref{sec:tre} we recall the extension technique used to define the fractional Gaussian perimeter of a measurable set (roughly speaking, we introduce a new ``vertical'' variable in order to study an equivalent degenerate local problem in the upper halfspace in one dimension more), we give some estimate of the rate of convergence of the extension to the original function and we prove a crucial result to obtain a dimension free constant in our Main Theorem. We also give an approximation of the Gaussian fractional perimeter of the halfspace, whose precise computation is not known up to our knowledge. Section \ref{sec:quattro} is more technical; here we collect some useful results that relate the asymmetry of a given measurable set with the asymmetry of some suitable level sets of the extension. Section \ref{sec:cinque} is devoted to the proof of the Main Theorem. Finally, in Section \ref{sec:sei}, we collect some remarks about our results and we discuss some open problems arising from our analysis.

\paragraph*{\bf Acknowledgements} A.C. has been partially supported by the TALISMAN project Cod. ARS01-01116. S.C. has been partially supported by the ACROSS project Cod. ARS01-00702. D.A.L. has been supported by the Academy of Finland grant 314227. 
D.P. is member of G.N.A.M.P.A. of the Italian Istituto Nazionale di Alta Matematica (INdAM) and has been partially supported by the PRIN 2015 MIUR project 2015233N54. 
\vspace{.5cm}
\section{Preliminary Results}\label{sec:intro}
For $N\in\N$ we denote by $\gamma_N$ and $\mathcal{H}^{N-1}_\gamma$, respectively, the Gaussian measure on $\R^N$ and the 
$(N-1)$-Hausdorff Gaussian measure
\begin{align*}
\gamma_N&:=\frac{1}{(2\pi)^{N/2}}e^{-\frac{|\cdot|^2}{2}}\Leb{N},
\\
\mathcal{H}^{N-1}_\gamma&:=\frac{1}{(2\pi)^{(N-1)/2}}e^{-\frac{|\cdot|^2}{2}}\mathcal{H}^{N-1},
\end{align*}
where $\Leb{N}$ and $\mathcal{H}^{N-1}$ are the Lebesgue measure and the Euclidean $(N-1)$-dimensional Hausdorff measure, respectively. 
When $k\in \{1,\dots, N\}$ is a given integer, we denote by $\gamma_k$ the standard $k$-dimensional Gaussian measure; when there is no ambiguity we simply write $\gamma$ instead of $\gamma_N$. 

The Gaussian perimeter of a measurable set $E$ in an open set $\Omega$ is defined as
$$P_\gamma(E;\Omega)=\sqrt{2\pi}\sup\left\{\int_E\left(\text{div}\,\varphi-\varphi\cdot x\right)\:d\gamma(x):\varphi\in C^\infty_c(\Omega;\R^N),\ \|\varphi\|_\infty\le 1\right\}.$$
If $\Omega=\R^N$, we denote the Gaussian perimeter of $E$ in the whole $\R^N$ simply by $P_\gamma(E)$. Moreover, if $E$ has finite Gaussian perimeter, then $E$ has locally finite Euclidean perimeter and it holds
$$
P_\gamma(E)=\mathcal{H}^{N-1}_\gamma(\partial^\ast E)=\frac{1}{(2\pi)^{\frac{(N-1)}{2}}}
\int_{\partial^\ast E}e^{-\frac{|x|^2}{2}}d\mathcal{H}^{N-1}(x),
$$
where $\partial^\ast E$ is the reduced boundary of $E$. We refer to \cite{AFP} for the properties of 
sets with finite perimeter. 

We introduce the increasing function $\Phi:\mathbb{R}\rightarrow (0,1)$ by 
$$
\Phi(r):=\int_{-\infty}^r d\gamma_1(t),
$$
and its inverse $\Phi^{-1}:(0,1)\rightarrow \mathbb{R}$.
We have 
$$
\gamma(H_{\omega,r})=\Phi(r)
$$
and
$$
P_\gamma(H_{\omega,r})=e^{-r^2/2},
$$
where, for $\omega\in\mathbb{S}^{N-1}$ and $r\in\mathbb{R}$, $H_{\omega,r}$ denotes the halfspace
$$
H_{\omega,r}:=\left\{x\in \mathbb{R}^N\quad\text{s.t.}\quad x\cdot\omega <r \right\}.
$$
Moreover, the Gaussian perimeter of any halfspace with Gaussian volume $m\in (0,1)$ is given by
\begin{equation}
\label{eq:isoperimetricfunction}
I(m):=e^{-\frac{\Phi^{-1}(m)^2}{2}},
\end{equation}
where $I:(0,1)\rightarrow(0,1]$ is usually called \emph{isoperimetric function}, and the Gaussian isoperimetric inequality reads as follows
\begin{equation}
\label{eq:gaussisopine}
P_\gamma(E) \ge I(\gamma(E)),
\end{equation}
stating that halfspaces are the unique (see \cite{CarKer}) volume constrained minimizers of the Gaussian perimeter. A sharp stability result for \eqref{eq:gaussisopine} has been obtained in \cite{BarBraJul}. 
Following \cite{EhrScand}, we introduce a suitable notion of symmetrization in the Gauss space. 
First, for any $J\subset\R$ we set 
\begin{equation}\label{defI*}
J^*=(-\infty, \Phi^{-1}(\gamma_1(J))) .
\end{equation}
Then, for $h\in \R^N$ with $|h|=1$, we consider the projection
$x'=x-(x\cdot h) h$ and write $x=x'+th$ with $t\in \R$,  
and for every measurable function $u:\R^N\to\R$ we define the symmetrized function in the sense of Ehrhard
\begin{equation}\label{defu*}
u_h^*(x'+th)=\sup\Bigl\{c\in {\mathbb R}:\ t\in\{u(x'+\cdot h)>c\}^*\Bigr\}. 
\end{equation}
Notice that if $u$ is (weakly) differentiable, $u^*_h$ is differentiable as well and the inequality
\[
\int_{\R^N} |\nabla u^*_h(x)|^2\, d\gamma(x) \leq \int_{\R^N}|\nabla u(x)|^2\, d\gamma(x)
\]
holds, see \cite[Theorem 3.1]{ehrhard} for the Lipschitz case; the Sobolev case easily follows by approximation. 
Since symmetrization preserves the class of characteristic functions, for every measurable set $E\subset \R^N$
we may define the Ehrhard-symmetrized set $E_h^*$ through the equality 
\[
\chi_{E_h^*}=(\chi_E)_h^* .
\]
We define the {\em Gaussian Fraenkel asymmetry} and the {\em fractional Gaussian isoperimetric deficit} of 
a set $E$ as
$$
\mathcal{A}_\gamma(E):=\min_{\omega\in\mathbb{S}^{N-1}}
\frac{\gamma(E\triangle H_{\omega,r})}{\gamma(E)},
$$
and
$$
D^{\gamma}_s(E):=P^\gamma_s(E)-P^{\gamma}_s(H_{\omega,r}),
$$
where $\triangle$ stands for the symmetric difference between sets and $P^\gamma_s(E)$ is the 
{\em $s$-fractional Gaussian perimeter} of $E$, see Section \ref{sec:tre}. 
These definitions are motivated 
by the fact that halfspaces are the optimal sets for the fractional isoperimetric problem as well, see 
\cite{NovPalSir}. 

\section{The extension technique and the fractional Gaussian perimeter}
\label{sec:tre}

In this section we collect the main results leading to the definition of the fractional Gaussian perimeter 
of a set and some preliminary results. Our approach is  based on the extension technique due to Caffarelli-Silvestre \cite{CafSil} in the 
Euclidean case and extended to wider frameworks, including the Gaussian case, by Stinga-Torrea in 
\cite{StiTor}. In the sequel, for any $1\le p<\infty$ we use the notation $L^p_\gamma$ for the space
$L^p(\R^N,d\gamma)$ and recall that in the Gaussian case the Ornstein-Uhlenbeck operator plays the same 
role as the Laplacian in the Euclidean setting. The Ornstein-Uhlenbeck operator $\Delta_\gamma$ is defined,
for $u$ sufficiently smooth, as
\begin{equation}
\label{eq:OU}
(\Delta_\gamma u)(x):=(\Delta u)(x)- x\cdot\nabla u(x).
\end{equation}
Since it comes from the symmetric bilinear form
\[
{\mathcal E}(u,v) :=\frac 12\int_{\R^N} \nabla u \cdot \nabla v\, d\gamma,
\]
we have that $-\Delta_\gamma$ is a positive definite selfadjoint operator which generates a $C_0$-semigroup 
of contractions, which we denote by $e^{t\Delta_{\gamma}}$, in $L^2_\gamma$ (see, \cite{LunMetPal} for a 
recent survey of the main properties of $\Delta_{\gamma},\ e^{t\Delta_{\gamma}}$ and references). As in \cite{StiTor}, we can define its fractional powers by means of classical spectral decomposition by the
Bochner's subordination formula (see e.g. \cite{MarSan})
\begin{equation}
\label{eq:bochnersubordination}
(-\Delta_\gamma)^su:=\frac{1}{\Gamma(-s)}\int_0^{\infty}\frac{e^{t\Delta_\gamma} u-u}{t^{s+1}}dt,
\end{equation}
where $\Gamma$ denotes the Euler gamma function and the Ornstein-Uhlenbeck semigroup $e^{t\Delta_\gamma}$ 
is given by the Mehler formula recalled in \cite{LunMetPal}
\begin{align*}
(e^{t\Delta_\gamma}u)(x)&:=\frac{1}{\left(2\pi(1-e^{-2t})\right)^{N/2}}\int_{\R^N}u(e^{-t}x-y)e^{-\frac{|y|^2}{2(1-e^{-2t})}}dy
\\
&=\int_{\R^N}u(e^{-t}x+\sqrt{1-e^{-2t}}y)d\gamma(y).
\end{align*}
Since for any $\lambda>0$ it holds 
$$
\left(\frac{1}{|\Gamma(-\frac s2)|}\int_0^{\infty}\frac{1-e^{-t\lambda}}{t^{\frac s2+1}}dt\right)^2=\lambda^s,
$$
again by functional calculus and Bochner's subordination formula we deduce 
\begin{equation}\label{eq:s/2}
(-\Delta_\gamma)^{\frac s2}\circ(-\Delta_\gamma)^{\frac s2}=(-\Delta_\gamma)^s.
\end{equation}

For an equivalent definition of $(-\Delta_\gamma)^s$ and for other qualitative properties involving the fractional Ornstein-Uhlenbeck operator we refer to \cite{FeStVo}.

The next proposition is an easy consequence of selfadjointness. 
\begin{proposition}\label{selfadjointness}
    For $u, v\in\text{\emph{Dom}}((-\Delta_\gamma)^s)$ it holds
\end{proposition}
\[
\int_{\R^N} e^{t\Delta_\gamma}(-\Delta _\gamma)^s v  u \, d\gamma = 
\int_{\R^N}e^{\frac t2\Delta_\gamma}(-\Delta _\gamma)^\frac s2 v e^{\frac t2\Delta_\gamma}
(-\Delta _\gamma)^\frac s2u\, d\gamma .
\]
\begin{proof}
Since $(-\Delta_\gamma)^s$ and $e^{t\Delta_\gamma}$ are selfadjoint operators in $L^2_\gamma$, from \eqref{eq:s/2} and the semigroup law we get 
\begin{align*}
&\int_{\R^N} e^{t\Delta_\gamma}(-\Delta _\gamma)^s v  u \, d\gamma = 
\int_{\R^N} e^{t\Delta_\gamma}(-\Delta _\gamma)^\frac s2
\circ (-\Delta _\gamma)^\frac s2 v  u \, d\gamma 
\\
=& 
\int_{\R^N}(-\Delta _\gamma)^\frac s2e^{\frac t2\Delta_\gamma} e^{\frac t2\Delta_\gamma}
(-\Delta_\gamma)^\frac s2v\, u\, d\gamma =
\int_{\R^N}e^{\frac t2\Delta_\gamma}(-\Delta _\gamma)^\frac s2 v e^{\frac t2\Delta_\gamma}
(-\Delta _\gamma)^\frac s2u\, d\gamma .
\end{align*}
\end{proof}
As pointed out by Stinga and Torrea in \cite{StiTor}, the fractional powers of the Ornstein-Uhlenbeck
operator can be obtained through an auxiliary problem, as it happens in the Euclidean case, 
see \cite{CafSil}.
\begin{theorem} \label{teo: extension}
Let $\varphi\in\text{\emph{Dom}}((-\Delta_\gamma)^s)$. The solution of the extension problem
\begin{equation}\label{eq:extension}
\left\{
\begin{array}{ll}
\Delta_{\gamma_x} V+\frac{1-2s}{z}\partial_zV+\partial^2_zV=0\quad&\text{in}\quad\R^{N+1}_{+} \\
V(x,0)=\varphi(x) &\text{in}\quad\R^N.
\end{array}\right.
\end{equation}
is given by 
\begin{equation}
U_\varphi (x,z) = \frac{1}{\Gamma(s)}
\int_0^\infty e^{t\Delta_\gamma}(-\Delta _\gamma)^s \varphi(x) \frac{e^{-\frac{z^2}{4t}}}{t^{1-s}}\,dt
\end{equation}
and it satisfies
$$
-\lim_{z\to 0^+}z^{1-2s} \partial_z U_\varphi (x,z)=K_{2s}(-\Delta_\gamma)^s \varphi(x),
$$
where
\begin{equation}
\label{eq:constant}
K_{2s}:=\frac{2s|\Gamma(-s)|}{4^s \Gamma (s)}.
\end{equation}
\end{theorem}
Coming to fractional Sobolev spaces, for $s\in(0,1)$ in the spirit of \cite{StiTor} we define the space $H_\gamma^s$ as 
the space of functions $u\in L^2_\gamma$ such that the following seminorm
$$
[u]^2_{H^s_\gamma}:= \inf \left\{\iint_{\R^{N+1}_+}\left(|\nabla_x v|^2+|\partial_z v|^2\right)z^{1-2s}
d\gamma(x)dz:\, v \in H^1_{\text{loc}} (\R^{N+1}_+), \, v(\cdot,0)=u\right\}
$$
is finite. If for a  function $u$ the infimum is achieved, the minimizer $U\in H_{\text{loc}}^1(\R^{N+1}_+)$ 
of the above functional is a weak solution of \eqref{eq:extension} with $u$ in place of $\varphi$.
In particular, when $u=\chi_E$ for some measurable set $E$, we define the {\em 
fractional Gaussian perimeter} of $E$ as 
$$
P^\gamma_s(E):=\frac 12[\chi_E]^2_{H^{\frac s2}_\gamma}.
$$
After this preparation we define an inner product in $H_\gamma^s$ by 
$$
\langle u,v \rangle_{H_\gamma^s} =K_{2s}\int_{\R^N}v(-\Delta_\gamma)^s u\, d\gamma = 
K_{2s}\int_{\R^N} u (-\Delta_\gamma)^s v \,d\gamma
$$
whenever $u,v \in \text{Dom}((-\Delta_\gamma)^s)$. This gives the equality
$$
[u]_{H_\gamma^s}^2= K_{2s}
\int_{\R^N}u(-\Delta_\gamma)^s u\, d\gamma .
$$
Note that when $s<1$, using Bochner's formula, we have 
\begin{equation}
\label{eq:seminormandoufrac}
[u]_{H_\gamma^s}^2= K_{2s}\int_{\R^N}u(-\Delta_\gamma)^s u\, d\gamma= K_{2s}\|(-\Delta_\gamma)^\frac s2 u\|^2_{L^2_\gamma}
\end{equation}
for every $u \in \text{Dom}((-\Delta_\gamma)^s)$.

Let us prove that the fractional Gaussian perimeter of a halfspace is the same in any dimension. 

\begin{proposition}\label{prop:halfspace}
	For $s\in(0,1)$ and $r\in\R$ we set
	$$H_r:=(-\infty,r)\quad\text{and}\quad H^N_r:=\left\{x\in\R^N:\ x_N<r\right\}.$$
	Then we have 
	$$
	P^\gamma_s(H^N_r)=P^{\gamma_1}_s(H_r),
	$$
	i.e., $P^\gamma_s(H^N_r)$ does not depend on the dimension $N$.
\end{proposition}
\begin{proof}
	Let $(y,z)\in\R^2_+$, let $v(y,z)$ be the solution of
	\begin{equation}
	\label{eq:onedimensionalsolution}
\left\{
\begin{array}{ll}
	\partial^2_yu-y\partial_yu+\frac{1-s}{z}\partial_zu+\partial^2_zu=0\quad &\text{in}\quad \R^2_+ \\
	u(y,0)=\chi_{H_r}(y) &\text{in}\quad\R,
	\end{array}\right.
	\end{equation}
	and consider
	\begin{equation}
	\label{eq:ndimensionalsolution}
	\left\{ \begin{array}{ll}
	\Delta_{\gamma_x}u+\frac{1-s}{z}\partial_zu+\partial^2_zu=0\quad &\text{in}\quad \R^{N+1}_+ \\
	u(x,0)=\chi_{H^N_r}(x) &\text{in}\quad\R^N.
	\end{array}\right.
	\end{equation}
	We prove that $w(x,z):=v(x_N,z)$ solves \eqref{eq:ndimensionalsolution}. Indeed, we have 
	\begin{equation}
	\label{eq:solofndim}
\Delta_{\gamma}w+\frac{1-s}{z}\partial_zw+\partial^2_zw = 
\partial^2_{x_N}v-x_N\partial_{x_N}v+\frac{1-s}{z}\partial_zv+\partial^2_zv=0,
	\end{equation}
	and
	\begin{equation}
	\label{eq:boundarycond}
	w(x,0)=v(x_N,0)=\chi_{H_r}(x_N)=\chi_{H^N_r}(x).
	\end{equation}
	Putting together \eqref{eq:solofndim} and \eqref{eq:boundarycond} we have that $w$ solves \eqref{eq:ndimensionalsolution}.
	Now we note that $w$ has finite energy. Indeed, 
	\begin{equation}
	\label{eq:finitenergy}
	\begin{split}
	\iint_{\R^{N+1}_+}\left(|\nabla_xw|^2+|\partial_zw|^2\right)d\gamma_N(x)z^{1-s}dz 
	=\iint_{\R^2_+}\left(|\partial_{y}v|^2+|\partial_zv|^2\right)d\gamma_1(y)z^{1-s}dz,
	\end{split}
	\end{equation}
	where we have used that $\gamma_N = \gamma_{N-1}\otimes \gamma_1$ and 
	$$
	\int_{\R^{N-1}}d\gamma_{N-1}(x')=1.
	$$
	Since the functional 
	$$
	H^1_{\text{loc}}\ni\varphi  \mapsto \iint_{\R^{N+1}_+}\left(|\nabla_x\varphi|^2+
	|\partial_z\varphi|^2\right)d\gamma_N(x)z^{1-s}dz 
	$$
	is strictly convex, it has only one critical point which coincides with the minimizer. Hence we 
	have proved that $w(x,z)=v(x_N,z)$ is the solution of the minimum problem 
	$$
	\inf \left\{\iint_{\R^{N+1}_+}\left(|\nabla_x u|^2+|\partial_z u|^2\right)d\gamma_N(x)z^{1-s}dz:
	\, u \in H^1_{\text{loc}} (\R^{N+1}_+), \, u(\cdot,0)=\chi_{H^N_r}\right\},
	$$
	and recalling the definition of $P_\gamma^s(H_r^N)$, the equality in \eqref{eq:finitenergy} gives the result. 
\end{proof}

\begin{remark}\label{rem:apperim}
	As it will be clear later, in order to have a more accurate control on the constant in the inequality \eqref{eq:main} we need an approximation of the value of the fractional Gaussian perimeter of the halfspace.
	Firstly, we define the normalized Hermite polynomials as
	$$
	h_n (x)= \frac{(-1)^n}{\sqrt{n!}}e^{\frac{x^2}{2}}
	\left(\frac{d}{dx}\right)^{n}\left(e^{-\frac{x^2}{2}}\right).
	$$
	It is well known that
	$$
	-\Delta_{\gamma_1} h_{n}= n h_n\quad\text{in}\quad\R \;\;\; \text{and}\;\;\;\int_{-\infty}^{+\infty}h_n h_m\, d\gamma= \delta_m^n.
	$$
	Thus now define the halfline $H_r:=(-\infty,r)$ and $f^r(x):=\chi_{H_r}(x)$. We expand $f^r$ on the basis given by $h_n$ and have
	$$
	f^r= \sum _{k=0}^{\infty} f^r_k h_k.
	$$
	It is quite simple to evaluate $f^r_k$, indeed those are just the projection of $f^r$ on $h_k$ and are given, for any $k\in\N\cup\left\{0\right\}$, by
	\[
	f^r_k= \int_{-\infty}^{+\infty} f^rh_k  d\, \gamma=
	\frac{1}{\sqrt{2\pi}}\int_{-\infty}^r \frac{(-1)^k}{\sqrt{k!}}
	\left(\frac{d}{dx}\right)^{k}\left(e^{-\frac{x^2}{2}}\right)\, dx
	=\frac{(-1)^k}{\sqrt{2\pi k!}}\left(\frac{d}{dr}\right)^{k-1}\left(e^{-\frac{r^2}{2}}\right),
	\]
	where, with abuse of notation when $k=0$
	$$
	f_0^r=\frac{1}{\sqrt{2\pi}}\left(\frac{d}{dr}\right)^{-1}\left(e^{-\frac{r^2}{2}}\right):=\frac{1}{\sqrt{2\pi}}\int_{-\infty}^re^{-\frac{t^2}{2}}dt.
	$$
	Hence the following formula holds
	\begin{equation}
	\begin{split}
	P_s^{\gamma_1} (H_r)&= \frac 12\int_{-\infty}^{+\infty} f^r (-\Delta_{\gamma_1})^\frac s2 f^r\, d \gamma=\frac 12\left(f_0^r\sum_{k=1}^\infty k^{\frac s2}f_k^r\int_{-\infty}^{+\infty}h_kd\gamma+\sum_{k=1}^\infty k^{\frac s2}(f_k^r)^2\right)
	\\
	&=\frac{1}{4\pi}\sum_{k=1}^\infty k^{\frac s2}\frac{1}{k!}\left(\left(\frac{d}{dr}\right)^{k-1}\left(e^{-\frac{r^2}{2}}\right)\right)^2=\frac{1}{4\pi}e^{-r^2}\sum_{k=1}^\infty \frac{1}{k^{1-\frac s2}}h_{k-1}^2(r),
	\end{split}
	\end{equation}
	where in the second and the third equality, respectively, we used the fact that
	$$
	\int_{-\infty}^{+\infty}h^2_kd\gamma=1\quad\text{and}\quad\int_{-\infty}^{+\infty}h_kd\gamma=0.
	$$
	Now we use the asymptotic behavior of the Hermite polynomials (see \cite[Pag. 201, Formula 18]{ErMaObTr}). After the change of variable $x=\frac {r}{\sqrt{2}}$ and the use of Stirling's formula for the Gamma function, we see that there exists $\nu\in\N$ such that
	$$
	h_{k-1}(r)\simeq \left(\frac 2\pi\right)^{1/4}\frac {e^{\frac{r^2}{4}}}{(k-1)^{\frac14}}\quad\text{for $k\ge\nu$}.
	$$
	Therefore,
	\begin{align*}
	P_s^{\gamma_1} (H_r)&\simeq \frac{1}{4\pi} \sqrt{\frac 2\pi}e^{-\frac{r^2}{2}}\left(\sum_{k=1}^\nu\frac{1}{k^{1-\frac s2}}h_{k-1}^2(r) +\sum_{k=\nu+1}^\infty\frac{1}{k^{1+\frac{1-s}{2}}}\right)\\
	&\simeq  \frac{1}{4\pi}\sqrt{\frac 2\pi}e^{-\frac{r^2}{2}}\left(c(r,\nu,s)+ \int_{\nu+1}^\infty \frac{dx}{x^{1+\frac{1-s}{2}}}\right)=\sqrt{\frac{\pi}{2}}\frac{1}{\pi^2}\left(c(r,\nu,s)+\frac{ e^{-\frac{r^2}{2}}(\nu+1)^{-\frac{1-s}{2}}}{1-s}\right)
	\end{align*}
	where $c(r,\nu,s)$ is the partial sum up to $k=\nu$ that is uniformly bounded with respect to $s\in[0,1]$ (since $\nu$ does not depend on $s$). Using Proposition \ref{prop:halfspace} this simply means that
	$$
	\lim_{s\to 1^-}(1-s)P^\gamma_s(H^N_r)=\lim_{s\to 1^-}(1-s)P^{\gamma_1}_s(H_r)\simeq \sqrt{\frac{\pi}{2}}\frac{1}{\pi^2}e^{-\frac{r^2}{2}}=\sqrt{\frac{\pi}{2}}\frac{1}{\pi^2}P_\gamma(H^N_r).
	$$
\end{remark}

From now on to shorten the notation, we set $U_E=U_{\chi_E}$ to denote the solution of problem \eqref{eq:extension}
when $\varphi=\chi_E$. 

The last proposition of this section gives an estimate of the rate of convergence of the 
Stinga-Torrea extension and will be useful later. 
\begin{proposition} \label{pr: spiccolo}
Let $s\in(0,1)$ and $\varphi\in\text{\rm Dom}((-\Delta_\gamma)^s)$. Let $U_\varphi$ be the solution of the extension problem
\begin{equation}
\left\{
\begin{array}{ll}
\Delta_{\gamma_x} V+\frac{1-2s}{z}\partial_z V+\partial^2_z V=0 \quad&\text{\rm in}\quad \R^{N+1}_+,
\\
V(x,0)=\varphi(x) &\text{\rm in}\quad\R^N.
\end{array}\right.
\end{equation}
Then, the following estimate holds
\begin{equation}\label{eq:2.11nostra}
\langle \varphi-U_\varphi(\cdot,z), \varphi\rangle_{L^2_\gamma}=\int_{\R^N} \varphi (\varphi-U_\varphi(\cdot,z))\,d\gamma\leq \beta_{2s} z^{2s} [\varphi]_{H^s_\gamma}^2
\end{equation}
with 
$$
\beta_{2s}:=\frac{1}{4^sK_{2s}}\frac{\Gamma(1-s)}{\Gamma(1+s)},
$$
where $K_{2s}$ is given in \eqref{eq:constant}.
\end{proposition}
\begin{proof}
As a consequence of Theorem \ref{teo: extension} we know that the solution $U_\varphi(x,z)$ is given by
$$
U_\varphi(x,z)=\frac{1}{\Gamma(s)}\int_0^{\infty}e^{t\Delta_\gamma}((-\Delta_\gamma)^s\varphi)(x)
\frac{e^{-\frac{z^2}{4t}}}{t^{1-s}}\:dt.
$$
Then, we can write 
$$
U_\varphi(x,z)-\varphi(x)=\frac{1}{\Gamma(s)}\int_0^{\infty}e^{t\Delta_\gamma}((-\Delta_\gamma)^s\varphi)(x)
\left(\frac{e^{-\frac{z^2}{4t}}-1}{t^{1-s}}\right)dt
$$
and using Proposition \ref{selfadjointness}
\begin{equation}\label{eq:est2.11.1}
\begin{split}
\langle \varphi&-U_\varphi(\cdot,z), \varphi\rangle _{L^2_\gamma}=\frac{1}{\Gamma(s)}\int_0^{\infty}dt\int_{\R^N}\varphi e^{t\Delta_\gamma}((-\Delta_\gamma)^s\varphi)\left(\frac{1-e^{-\frac{z^2}{4t}}}{t^{1-s}}\right)\,d\gamma\\
&=\frac{1}{\Gamma(s)}\int_0^{\infty}\frac{1-e^{-\frac{z^2}{4t}} }{t^{1-s}}dt\int_{\R^N}e^{\frac t2\Delta_\gamma}((-\Delta_\gamma)^\frac s2\varphi) e^{\frac t2\Delta_\gamma}((-\Delta_\gamma)^\frac s2\varphi)\,d\gamma. 
\end{split}
\end{equation}
Now recall that the function $v(\cdot,t)=e^{\frac{t\Delta_\gamma}{2}}((-\Delta_\gamma)^\frac s2\varphi)$ is nothing 
but the solution of the Cauchy problem
\begin{equation}
\left\{
\begin{array}{ll}
2 \partial_t v= \Delta_\gamma v \quad &(x,t) \in \R^N\times(0,\infty)
\\
v(x,0)= (-\Delta_\gamma)^\frac s2\varphi(x) & x\in\R^N
\end{array}\right.
\end{equation}
evaluated at $t$. Hence we have
\begin{equation}
\label{eq:nonincreasing}
\frac{d}{dt}\|v(\cdot,t)\|^2_{L^2_\gamma}= \int_{\R^N}v\Delta_\gamma v\, d\gamma=
\int_{\R^N} v\Div (e^{-\frac{|x|^2}{2}}\nabla v)\,dx= - \int_{\R^N}|\nabla v|^2\, d\gamma \leq 0
\end{equation}
which implies that the $L^2_\gamma$ norm is nonincreasing in the $t$ variable. Hence, using \eqref{eq:seminormandoufrac} and \eqref{eq:nonincreasing} formula  \eqref{eq:est2.11.1} can be rewritten as
\begin{equation}\label{eq:est2.11.2}
\begin{split}
\langle \varphi-U_\varphi(\cdot,z), \varphi\rangle _{L^2_\gamma}
& \le  \frac{1}{\Gamma(s)}\|(-\Delta_\gamma)^\frac s2 \varphi\|_{L_\gamma^2}^2  \int_0^\infty 
\frac{1-e^{-\frac{z^2}{4t}}}{t^{1-s}}\,dt 
\\
& = \frac{1}{4^s}\frac{\Gamma(1-s)}{\Gamma(1+s)}z^{2s}\|(-\Delta_\gamma)^\frac s2 \varphi\|_{L_\gamma^2}^2 
= \beta_{2s}z^{2s} [\varphi]^2_{H^s_\gamma},
\end{split}
\end{equation}
with $\beta_{2s}$ as in the statement, and the proof is complete. 
\end{proof}

Since we are interested in applying the above lemma to characteristic functions and fractional perimeters, 
it is convenient to rewrite the above lemma with $\varphi=\chi_E$ and $s$ replaced by $s/2$. 
We notice that if $\varphi$ is a characteristic function, then $U_\varphi\leq 1$ everywhere 
(to prove it one uses the variational formulation and shows that replacing any competitor $v$ 
with $\max \{v,1\}$ the energy does not increase). This observation allows us to say that $\chi_{E} (\chi_E-U_{E})\geq 0$ in the whole $\R^{N+1}_+$.
Then \eqref{eq:2.11nostra} reads 
\begin{equation}\label{eq:quantichar}
\int_{E} (1-U_{E}(\cdot,z))\,d\gamma \leq \beta_{s} z^{s} [\chi_E]^2_{H_\gamma^\frac s2}= 
2 \beta_{s} z^{s} P_s^\gamma (E).
\end{equation}

\section{Estimates on the level sets of the extension}
\label{sec:quattro}
This section contains some technical results that are the core of the proof of the Main Theorem. Our strategy follows the ideas in \cite{BraCinVit}: we first estimate $D_s^\gamma(E)$ from below with a quantity involving the asymmetry of the superlevel sets of $U_E(\cdot,z)$ and then, in a suitable range of values for the function $U_E$ and for the vertical variable $z$, we show that the asymmetry of the superlevel sets is estimated from below by $\mathcal{A}_\gamma(E)$.

\vspace{.5cm}

The following proposition provides an enhanced version of an inequality proved in \cite{NovPalSir}.
In the spirit of \cite{BraCinVit, FuMiMo}, given a set $E$, we apply the Stinga-Torrea extension to the
function $\chi_E$ and exploit the sharp Gaussian quantitative inequality proved in \cite{BarBraJul}.
\begin{proposition}
	\label{prop:quattroquattronostra}
	Let $s\in(0,1)$ and let $E\subset\R^N$ be an open set with 
	$P_s^{\gamma}(E)<\infty$. For $t>0$ and $z>0$, we set
	$$
	E_{t,z}:=\left\{x\in\R^N: U_E(x,z)>t \right\}\quad,\quad \mu_z(t):=\gamma(E_{t,z}),
	$$
	and, for any $m\in(0,1)$
	$$f(m):=\frac{e^{\frac{\Phi^{-1}(m)^2}{2}}}{1+\Phi^{-1}(m)^2}.$$
Then for every halfspace $H:=H_{\omega,r}$ s.t. $\gamma(H)=\gamma(E)$ we have
	\begin{equation}
	P_s^\gamma(E)-P_s^\gamma(H)\ge\frac{1}{2c}\int_0^{\infty} z^{1-s} dz
	\int_0^{\infty}f(\mu_z(t))\mathcal{A}^2_\gamma(E_{t,z})\frac{I(\mu_z(t))}{-\mu^{'}_z(t)}dt
	\end{equation}
	where $c$ is the absolute constant in \cite[Main Theorem]{BarBraJul}. 
\end{proposition}
\begin{proof}
We have 
$$
P_s^\gamma(E)=\frac 12\left[\chi_E\right]^2_{H_\gamma^{\frac{s}{2}}}=
\frac 12\left(\iint_{\R^{N+1}_+}z^{1-s}|\nabla_xU_E|^2d\gamma(x)dz+
\iint_{\R^{N+1}_+}z^{1-s}|\partial_z U_E|^2d\gamma(x)dz\right).
$$
For the $z$-derivative, we may compute (see \cite[Lemma 3.2]{NovPalSir}).
\begin{equation}\label{eq:novpalsir}
\iint_{\R^{N+1}_+}z^{1-s}|\partial_z U_E|^2d\gamma(x)dz \ge 
\iint_{\R^{N+1}_+}z^{1-s}|\partial_z U^\ast_E|^2d\gamma(x)dz,
\end{equation}
while for the $x$-derivative, by using the coarea formula we have 
\begin{equation}
\label{eq:xder}
\begin{split}
\iint_{\R^{N+1}_+}z^{1-s}&|\nabla_xU_E|^2d\gamma(x)dz
=\int_0^{\infty}z^{1-s}dz\int_0^{\infty}dt\int_{\left\{x\in\mathbb{R}^N: U_E(x,z)=t\right\}}
|\nabla_xU_E|d\mathcal{H}_\gamma^{N-1}(x) 
\\
&\ge \int_0^{\infty}z^{1-s}dz\int_0^{\infty}\frac{P_\gamma(E_{t,z})^2}{\int_{\left\{x\in\mathbb{R}^N: U_E(x,z)=t\right\}}\frac{d\mathcal{H}_\gamma^{N-1}(x)}{|\nabla_xU_E|}}dt,
\end{split}
\end{equation}
where we have used H\"older's inequality with exponents $(2,2)$ to get
\begin{equation}\label{eq:holdfurba}
P_\gamma(E_{t,z})^2
\le \left(\int_{\partial^\ast E_{t,z}}|\nabla_x U_E|\:d\mathcal{H}^{N-1}_\gamma(x)\right)
\left(\int_{\partial^\ast E_{t,z}}\frac{d\mathcal{H}^{N-1}_\gamma(x)}{|\nabla_x U_E|}\right).
\end{equation}
Now, we consider the Ehrhard-symmetrized of the set $E_{t,z}$
$$
E_{t,z}^\ast=\left\{x\in\R^N:\quad U^\ast_E(x,z)>t \right\}
$$
and, from the trivial inequality
$$
(P_\gamma(E_{t,z})-P_\gamma(E^\ast_{t,z}))^2\ge 0,
$$
we easily obtain 
\begin{equation}
\label{eq:square}
P_\gamma(E_{t,z})^2 \ge P_\gamma(E^\ast_{t,z})^2+2P_\gamma(E^\ast_{t,z})(P_\gamma(E_{t,z})-P_\gamma(E^\ast_{t,z})).
\end{equation}
Moreover the Main Theorem in \cite{BarBraJul} provides us the following quantitative inequality
\begin{equation}
\label{eq:cifumapr}
P_\gamma(E)-P_\gamma(E^\ast) = P_\gamma(E)-e^{-\frac{r^2}{2}} \ge 
\frac{e^{\frac{r^2}{2}}}{4c(1+r^2)} \mathcal{A}_\gamma(E)^2,
\end{equation}
for any set $E$ such that $\gamma(E)=m$, with $r=\Phi^{-1}(m)$, and for some absolute constant $c>0$, see the discussions in the 
Introduction of \cite{BarBraJul} and in \cite{BarJul}.

Inserting \eqref{eq:cifumapr} in \eqref{eq:square} we conclude that
\begin{equation}
\label{eq:quantsquared}
P_\gamma(E_{t,z})^2\ge P_\gamma(E^\ast_{t,z})^2+\frac{f(\mu_z(t))}{2c}P_\gamma( E^\ast_{t,z})\mathcal{A}_\gamma(E_{t,z})^2.
\end{equation}
If we put \eqref{eq:quantsquared} into \eqref{eq:xder} we obtain 
\begin{equation}\label{eq:primaquantitativa}
\begin{split}
\iint_{\R^{N+1}_+}z^{1-s}|\nabla_xU_E|^2d\gamma(x)dz 
\ge & \int_0^{\infty}z^{1-s}dz\int_0^{\infty}\frac{P(E^\ast_{t,z})^2}{-\mu^{'}_z(t)}dt
\\
&+\frac{1}{2c}\int_0^{\infty}z^{1-s}dz\int_0^{\infty}f(\mu_z(t))\frac{P_\gamma(E^\ast_{t,z})\mathcal{A}_\gamma(E_{t,z})^2}{-\mu^{'}_z(t)}dt
\end{split}
\end{equation}
where we have the equalities
\begin{align*}
\mu_z(t)&=\gamma(E^*_{t,z})=
\int_t^{\infty}ds\int_{\partial E^\ast_{s,z}}\frac{d\mathcal{H}^{N-1}_\gamma(x)}{|\nabla_xU_E^\ast|},
\\
\mu'_z(t)&=-\int_{\partial E^\ast_{t,z}}\frac{d\mathcal{H}^{N-1}_\gamma(x)}{|\nabla_xU_E^\ast|}.
\end{align*}
By using these facts we obtain
\begin{align*}
\int_0^{\infty}z^{1-s}dz&\int_0^{\infty}\frac{P(E^\ast_{t,z})^2}{-\mu^{'}_z(t)}dt=\int_0^{\infty}z^{1-s}dz\int_0^{\infty}\frac{P(E^\ast_{t,z})^2}{\int_{\partial E^\ast_{t,z}}\frac{d\mathcal{H}^{N-1}_\gamma(x)}{|\nabla_xU_E^\ast|}}dt\\
&=\int_0^{\infty}z^{1-s}dz\int_0^{\infty}\left(\int_{\partial E^\ast_{t,z}}|\nabla_xU_E^\ast|d\mathcal{H}^{N-1}_\gamma(x)\right)dt,
\end{align*}
where we have applied H\"older's inequality with exponents (2,2) as in \eqref{eq:holdfurba}. 
In this case the equality occurs, as the functions 
$|\nabla_xU_E^\ast|^{1/2}$ and 
$|\nabla_xU_E^\ast|^{-1/2}$ are constant on the level plane $\partial E^\ast_{t,z}$. By applying the coarea formula we get
\begin{equation}\label{eq:coareaultima}
\int_0^{\infty}z^{1-s}dz\int_0^{\infty}\left(\int_{\partial E^\ast_{t,z}}|\nabla_xU_E^\ast|d\mathcal{H}^{N-1}_\gamma(x)\right)dt=\iint_{\R^{N+1}_+}z^{1-s}|\nabla_xU_E^\ast|^2d\gamma(x)dz.
\end{equation}
By plugging \eqref{eq:coareaultima} into \eqref{eq:primaquantitativa} and summing with \eqref{eq:novpalsir} we finally obtain
\begin{align*}
P_\gamma^s(E)=&\frac 12\left(\iint_{\R^{N+1}_+}z^{1-s}|\nabla_xU_E|^2d\gamma(x)dz
+\iint_{\R^{N+1}_+}z^{1-s}|\partial_z U_E|^2d\gamma(x)dz\right)
\\
\ge& \frac 12\left(\iint_{\R^{N+1}_+}z^{1-s}|\nabla_xU^*_E|^2d\gamma(x)dz+
\iint_{\R^{N+1}_+}z^{1-s}|\partial_z U^*_E|^2d\gamma(x)dz\right)
\\
&+\frac{1}{2c}\int_0^{\infty}z^{1-s}dz\int_0^{\infty}f(\mu_z(t))\frac{P_\gamma(E^\ast_{t,z})\mathcal{A}_\gamma(E_{t,z})^2}{-\mu^{'}_z(t)}dt
\\
=&P_\gamma^s(H)+\frac{1}{2c}\int_0^{\infty}z^{1-s}dz\int_0^{\infty}f(\mu_z(t))\frac{P_\gamma(E^\ast_{t,z})\mathcal{A}_\gamma(E_{t,z})^2}{-\mu^{'}_z(t)}dt,
\end{align*}
hence, recalling that $P_\gamma(E^\ast_{t,z})=I(\gamma(E^\ast_{t,z}))$, we get the thesis.
\end{proof}

The next lemma roughly says that if we know how asymmetric is a set and we are given another set which is not too different (in the measure sense) from the first one, then the asymmetry of the second set can be controlled from below by the asymmetry of the first one. 
\begin{lemma}
	\label{lem:quattrounonostro}
	Let $E, F\subset\R^N$ be two measurable sets such that
	\begin{equation}
	\label{eq:transferofasym}
	\frac{\gamma(F\triangle E)}{\gamma(F)}\leq \kappa\mathcal{A}_\gamma(F),
	\end{equation}
	for some $0<\kappa<1/2$. Then
	$$
	\mathcal{A}_\gamma(E) \ge \frac{1-2\kappa}{c_\kappa}\mathcal{A}_\gamma(F),
	$$
	where $c_\kappa:=\begin{cases}
	1,\quad\text{if}\quad \gamma(E\setminus F)=0, \\
	1+2\kappa,\quad\text{if}\quad \gamma(E\setminus F)>0.
	\end{cases}$
\end{lemma}
\begin{proof}
	The case $\mathcal{A}_\gamma(F)=0$ is trivial, so we can suppose that $\mathcal{A}_\gamma(F)>0$. 
	We take a halfspace $H$ such that $\gamma(H)=\gamma(E)$ and
	$$
	\mathcal{A}_\gamma(E)=\frac{\gamma(E\triangle H)}{\gamma(E)},
	$$
	and the halfspace $H'$ with $\gamma(H')=\gamma(F)$ and such that $H$ is contained in $H'$ or vice versa. 
	We recall that
	$$
	\gamma(F\triangle E)=\left\|\chi_F-\chi_E\right\|_{L^1_\gamma},
	$$
	and by using the triangle inequality we obtain 
	\begin{equation*}
	\begin{split}
	\mathcal{A}_\gamma(E)&=\frac{\gamma(E\triangle H)}{\gamma(E)} \ge 
	\frac{\gamma(F)}{\gamma(E)}\left(\frac{\gamma(F\triangle H')}{\gamma(F)}
	-\frac{\gamma(H'\triangle H)}{\gamma(F)}
	-\frac{\gamma(F\triangle E)}{\gamma(F)}\right) 
	\\
	&\ge\frac{\gamma(F)}{\gamma(E)}\left(\mathcal{A}_\gamma(F)
	-2\frac{\gamma(F\triangle E)}{\gamma(F)}\right) 
	\ge\frac{\gamma(F)}{\gamma(E)}(1-2\kappa)\mathcal{A}_\gamma(F),
	\end{split}
	\end{equation*}
	where in the second inequality we have used the fact that
	$$
	\gamma(H'\triangle H)=\left|\gamma(F)-\gamma(E)\right|\le\gamma(F\triangle E).
	$$
	In order to conclude, we need to get a lower bound for the ratio $\gamma(F)/\gamma(E)$. 
	If $\gamma(E\setminus F)=0$, we have 
	$$
	\frac{\gamma(F)}{\gamma(E)}=\frac{\gamma(F)}{\gamma(E\cap F)}\ge 1.
	$$
	If $\gamma(E\setminus F)>0$, we observe that
	$$
	\frac{\gamma(F)}{\gamma(E)}=\frac{\gamma(F)}{\gamma(E\setminus F)+\gamma(E\cap F)}
	\ge\frac{\gamma(F)}{\gamma(F\triangle E)+\gamma(F)}\ge\frac{1}{1+\kappa\mathcal{A}_\gamma(F)}.
	$$
  We conclude by recalling that the Gaussian Fraenkel asymmetry is always smaller than 2.
	\end{proof}
	Now we prove a technical result similar to \cite[Lemma 4.2]{BraCinVit}. It states that if we are 
	not going too far in the vertical direction, then the level sets of the extension of the 
	characteristic function of a set $E$ are comparable to $E$ itself.
\begin{lemma}
	\label{lem:quattroduenostro}
	For $\alpha>0$ fixed, the following implication holds:
	$$
	\text{if}\quad \frac 14 \leq t \leq \frac 34 \quad \text{and} \quad 
0<z<\left(\frac{1}{8\alpha\beta_s P^\gamma_s(E)}\right)^{\frac{1}{s}},
	$$
	then
	\begin{equation}	\label{eq:estimatedifferencesets1}
	\gamma (E \setminus	\{x\in\R^N:\,U_E(x,z)>t \}) \leq \frac1\alpha
	\end{equation}
	and
	\begin{equation}
	\label{eq:estimatedifferencesets2}
	\gamma\left(\left\{x\in\R^N:\,U_E(x,z)>t \right\}\setminus E\right)
	\le \frac{1}{\alpha}.
	\end{equation}
\end{lemma}
\begin{proof}
Fixed $z\in(0,\infty)$, we set
	$$
	B_{E,z}:=\left\{x\in E: (1-U_E(x,z)) >
	2\beta_sP_s^\gamma(E)\alpha z^{s}\right\}.
	$$
	Then, by using the Markov-Chebychev inequality and \eqref{eq:quantichar}, we get 
	\begin{equation} \label{eq:lemmatecn}
  	\gamma\left(B_{E,z}\right)\leq 
	\frac{1}{2\beta_sP_s^\gamma(E)\alpha z^{s}}
	\int_{E}(1-U_E(\cdot,z))\, d\gamma
	\leq \frac{1}{\alpha }.
	\end{equation}
	We now take $t$ and $z$ as in the statement. Then for every $x\in E$ such that $U_E(x,z)\leq t$, we have
	$$
	1-U_E(x,z)\geq 1-t\ge\frac 14>2\alpha \beta_sP^\gamma_s(E)z^{s}
	$$
	that is
	\begin{equation*}
	\left\{x\in\R^N\,:\,U_E(x,z)\leq t \right\}\cap E =
	E \setminus	\left\{x\in\R^N\,:\,U_E(x,z)>t \right\}
	\subset B_{E,z}.
	\end{equation*}
	By using \eqref{eq:lemmatecn}, we get \eqref{eq:estimatedifferencesets1}. 
	Inequality \eqref{eq:estimatedifferencesets2} can be obtained in the same way replacing $E$ with $E^c$
	and using $U_{E^c}=1-U_E$.
	\end{proof}
	Next proposition is an easy application of the previous Lemmas \ref{lem:quattrounonostro} and \ref{lem:quattroduenostro} and is one of the main ingredients 
	in the proof of our Main Theorem.
\begin{proposition}
	\label{prop:quattrotrenostra}
	For $t\in\left[\frac 14,\frac 34\right]$ and $z\in(0,z_0]$, where
	$$
	z_0:=\left(\frac{\mathcal{A}_\gamma(E)\gamma(E)}{72\beta_sP^\gamma_s(E)}\right)^{\frac{1}{s}},
	$$ 
	we have
	\begin{equation}
	\label{eq:trequattordici}
	    |\gamma(E_{t,z}) - \gamma(E)| \leq \frac  {2}{9} \gamma(E)\mathcal{A}_\gamma(E)
	\end{equation}
and
	\begin{equation}
	\label{eq:asymmetry}
	\mathcal{A}_\gamma\left(E_{t,z}\right)\ge
	\frac {5}{13}\mathcal{A}_\gamma(E).
	\end{equation}
\end{proposition}
\begin{proof}
     Observe that by using \eqref{eq:estimatedifferencesets1} and  \eqref{eq:estimatedifferencesets2} in Lemma \ref{lem:quattroduenostro} with the choice
	$$
	\alpha:=\frac{9}{\mathcal{A}_\gamma(E)\gamma(E)},
	$$
	we get
	\begin{equation*}
	\begin{split}
	\frac{\gamma\left(E_{t,z}\triangle E\right)}{\gamma(E)}&=\frac{\gamma\left(E\setminus E_{t,z}\right)}{\gamma(E)}+\frac{\gamma\left(E_{t,z}\setminus E\right)}{\gamma(E)} \\
	&\le \frac{2}{\alpha}\frac{1}{\gamma(E)}=\frac 29\mathcal{A}_\gamma(E).
	\end{split}
	\end{equation*}
	Finally, by triangle inequality we have
	$$
	\gamma(E)-\gamma\left(E_{t,z}\triangle E\right)\le\gamma(E_{t,z})\le \gamma(E)+\gamma\left(E_{t,z}\triangle E\right),
	$$
	thus by joining the last two estimates we get \eqref{eq:trequattordici}.
	We can now apply Lemma \ref{lem:quattrounonostro} with $\kappa=2/9$, so we obtain
	$$
	\mathcal{A}_\gamma\left(E_{t,z}\right)\ge\frac{1-\frac 49}{1+\frac 49}\mathcal{A}_\gamma(E)=
	\frac {5}{13}\mathcal{A}_\gamma(E),
	$$
	and this concludes the proof.
\end{proof}
\section{Proof of the Main Theorem}
\label{sec:cinque}
Now our goal is to prove that
\begin{equation}
\label{eq:quantisoine}
D^\gamma_s(E)=P^\gamma_s(E)-P^\gamma_s(H)\ge C_{s,m}\mathcal{A}_\gamma(E)^{\frac{2}{s}}
\end{equation}
where $H$ is a halfspace such that $\gamma(H)=\gamma(E)=m$.
We also observe that if $P^\gamma_s(E)>2P^\gamma_s(H)$, then by using that $\mathcal{A}_\gamma(E)<2$
$$
P^\gamma_s(E)-P^\gamma_s(H)>P^\gamma_s(H)>\frac{P^\gamma_s(H)}{2^{\frac 2s}}\mathcal{A}_\gamma(E)^{\frac{2}{s}}.
$$
Therefore, we reduce ourselves to consider the case
\begin{equation}
\label{eq:assumpdoubleperimeter}
P^\gamma_s(E)\le2P^\gamma_s(H).
\end{equation}
We are now ready to prove our Main Theorem.
\begin{proof}[Proof of the Main Theorem]
Since $\gamma(E)+\gamma(E^c)=1$ and $P^\gamma_s(E)=P^\gamma_s(E^c)$ we can assume with no loss of generality that $\gamma(E)\le\frac 12$.

We set
$$
z_1:=\left(\frac{\mathcal{A}_\gamma(E)\gamma(E)}{144\beta_sP^\gamma_s(H)}\right)^{\frac{1}{s}},
$$
by assumption \eqref{eq:assumpdoubleperimeter}, we have
\begin{equation*}
z_1<z_0=\left(\frac{\mathcal{A}_\gamma(E)\gamma(E)}{72\beta_sP^\gamma_s(E)}\right)^{\frac{1}{s}},
\end{equation*}
where $z_0$ is defined in Proposition \ref{prop:quattrotrenostra}.
By using Proposition \ref{prop:quattroquattronostra} in conjunction with Proposition \ref{prop:quattrotrenostra}, we have 
\begin{equation*}
\begin{split}
P^\gamma_s(E)-P^\gamma_s(H)&\ge \frac{1}{2c}\int_0^{\infty}z^{1-s}dz\int_0^{\infty}f(\mu_z(t))\mathcal{A}_\gamma(E_{t,z})^2\frac{I(\mu_z(t))}{-\mu'_z(t)}dt \\
&\ge\frac{1}{2c}\int_0^{z_1}z^{1-s}dz\int_{\frac 14}^{\frac 34}f(\mu_z(t))\mathcal{A}_\gamma(E_{t,z})^2\frac{I(\mu_z(t))}{-\mu'_z(t)}dt \\
&\ge\frac{25}{338c}\mathcal{A}_\gamma(E)^2\int_0^{z_1}z^{1-s}dz\int_{\frac 14}^{\frac 34}f(\mu_z(t))\frac{I(\mu_z(t))}{-\mu'_z(t)}dt \\
&\ge\frac{25\sqrt{e}}{676c}\mathcal{A}_\gamma(E)^2\int_0^{z_1}z^{1-s}dz\int_{\frac 14}^{\frac 34}\frac{I(\mu_z(t))}{-\mu'_z(t)}dt.
\end{split}
\end{equation*}
where in the last inequality we used the fact that the function $\R\ni x\mapsto e(x):=\frac{e^{\frac{x^2}{2}}}{1+x^2}$ is bounded from below by $\sqrt{e}/2$ and that $f=e\circ\Phi^{-1}$.
We observe that by using \eqref{eq:trequattordici} and the fact that $\mathcal{A}_\gamma(E)<2$, for every
$t\in\left[\frac 14,\frac 34\right]$ we get 
$$
\frac 59\gamma(E)<\gamma(E)\left(1-\frac 29\mathcal{A}_\gamma(E)\right)\le\mu_z(t)\le\gamma(E)\left(1+\frac 29\mathcal{A}_\gamma(E)\right)<\frac{13}{9}\gamma(E),
$$
and so,
$$
I\left(\mu_z(t)\right)\ge \min\left\{I(\xi),\quad \xi\in\left[\frac 59\gamma(E),\frac{13}{9}\gamma(E)\right]\right\}=:\sigma_{\gamma(E)},
$$
for every $t\in\left[\frac 14,\frac 34\right]$ and for every $z\in[0,z_1]$.
This in turn implies that
$$
P_s^{\gamma}(E)-P_s^{\gamma}(H)\ge \frac{25\sqrt{e}}{676c}\sigma_{\gamma(E)}\mathcal{A}_\gamma(E)^2\int_0^{z_1}z^{1-s}dz\int_{\frac 14}^{\frac 34}\frac{1}{-\mu^{'}_z(t)}dt.
$$
We estimate the inner integral in $t$ by using Jensen's inequality
$$
\int_{\frac 14}^{\frac 34}\frac{1}{-\mu'_z(t)}dt\ge\frac{1}{4}
\left(\int_{\frac 14}^{\frac 34}-\mu^{'}_z(t)dt\right)^{-1}\ge
\frac{1}{4}\left(\gamma\left(E_{\frac 14,z}\right)-\gamma\left(E_{\frac 34,z}\right)\right)^{-1}.
$$
By using \eqref{eq:trequattordici} with $t=1/4$ and $t=3/4$, we get
$$
\gamma\left(E_{\frac 14,z}\right)-\gamma\left(E_{\frac 34,z}\right)\le\gamma(E)\left(1+\frac 29\mathcal{A}_\gamma(E)\right)-\gamma(E)\left(1-\frac 29\mathcal{A}_\gamma(E)\right)=\frac 49\gamma(E)\mathcal{A}_\gamma(E).
$$
In conclusion, we get
\begin{equation*}
\begin{split}
P_s^{\gamma}(E)-P_s^{\gamma}(H)&\ge\frac 94\frac{25\sqrt{e}}{676c}\frac{\mathcal{A}_\gamma(E)}{\gamma(E)}\frac{\sigma_{\gamma(E)}}{4}\int_0^{z_1}z^{1-s}dz \\
&=\frac{3^2\cdot 5^2}{676c}\frac{\mathcal{A}_\gamma(E)}{\gamma(E)}\frac{\sigma_{\gamma(E)}}{16}\frac{\sqrt{e}}{2-s}z_1^{2-s} \\
&=\frac{3^{4-\frac 4s}\cdot5^2}{13^2c}\left(\frac 12\right)^{\frac{8}{s} +2}\frac{\sqrt{e}}{2-s}\frac{\sigma_{\gamma(E)}\gamma(E)^{\frac 2s-2}}{\left(\beta_sP^\gamma_s(H)\right)^{\frac 2s-1}}\mathcal{A}_\gamma(E)^{\frac 2s},
\end{split}
\end{equation*}
and this concludes the proof.
\end{proof}

\section{Further remarks and open problems}
\label{sec:sei}
Some comments on the constant $C_{s,m}$ obtained in the Main Theorem are in order: though it is quite explicit, unfortunately we only have an upper bound for the constant $c$ (coming from the sharp quantitative Gaussian isoperimetric inequality in \cite{BarBraJul}) and we have only an approximation of the value of the fractional Gaussian perimeter of the halfspace provided by Remark \ref{rem:apperim}. Moreover, the constant does not seem to be stable as $s\to 0^+$ or $s\to 1^-$ and the exponent $2/s$ of the asymmetry does not seem to be sharp. Indeed, in complete similarity with the Euclidean case proved in \cite{FiFuMaMiMo}, we expect the optimal power to be $2$ for any $s\in(0,1)$ although the techniques we used do not lead to the expected sharp exponent even in the Euclidean case, as one can see in \cite{FuMiMo} for the fractional perimeter or in \cite{BraCinVit} for a nonlocal spectral functional.

The fact that $C_{s,m}$ is independent of the dimension suggests to generalize the result in infinite dimension, as usual in the framework of Gauss spaces, replacing $\R^N$ with an infinite dimensional Wiener space. Unfortunately, at the moment this is not possible using an argument of approximation via cylindrical functions, even in the local case. Indeed, the proof of our result relies on other papers where dimension-free inequalities are
provided, such as \cite{BarBraJul,BarJul}. Nevertheless, these results (as well as ours) do not
extend to the infinite dimensional case because fine properties of sets with finite perimeter and
regularity results for almost minimizers of the perimeter functional are used, that are not
available in infinite dimensions.

\end{document}